\documentclass[11pt, reqno]{amsart}

\usepackage{amssymb,amsmath,amscd,graphicx,
latexsym,amsthm}
\usepackage{amssymb,latexsym,amsmath,amscd,graphicx}
\usepackage[mathscr]{eucal}
 \input xy

\usepackage{hyperref}

 \xyoption{all}

\def\OO{{\mathcal O}}

\def\F{\mathcal{F}}

\def\G{\mathcal{G}}

\def\Pic0{{\rm Pic}^0}

\def\Aut0{{\rm Aut}^0}

\def\Bs{\mathrm{Bs}\,}
\def\PBs{\mathrm{PBs}\,}

\def\DD{{\mathbf{D}}}

\def\cX{\mathcal{X}}

\def\*{{\underline *}}

\def\Alb{{\rm Alb}\,}

\theoremstyle{plain}

\newtheorem{theorem}{Theorem}[section]

\newtheorem{proposition/example}[theorem]{Proposition/Example}
\newtheorem{proposition}[theorem]{Proposition}

\newtheorem{corollary}[theorem]{Corollary}
\newtheorem{lemma}[theorem]{Lemma}

\theoremstyle{definition}
\newtheorem{definition}[theorem]{Definition}

\newtheorem{remark}[theorem]{Remark}

\newtheorem{conjecture}[theorem]{Conjecture}
\newtheorem{conjecture/question}[theorem]{Conjecture/Question}

\newtheorem{remark/definition}[theorem]{Remark/Definition}
\newtheorem{notation/assumptions}[theorem]{Assumptions/Notation}

\numberwithin{equation}{section}

 \theoremstyle{remark}

\calclayout 
\begin{document}
\title[Paracanonical base locus, Albanese morphism,\,\ldots]{
Paracanonical base locus, Albanese morphism, and semi-orthogonal indecomposability of derived categories
}
 \author{Federico Caucci}

\address{Dipartimento di Matematica ``Federigo Enriques'', Universit\`a degli Studi di Milano,
Via Cesare Saldini 50, 20133 Milano -- Italy}
\curraddr{Dipartimento di Matematica e Informatica, Universit\`a di Ferrara, Via Machiavelli 30, 44121 Ferrara, Italy} 
 \email{federico.caucci@unife.it }
 \thanks{The author was supported by the ERC Consolidator Grant ERC-2017-CoG-771507-StabCondEn.
}

\maketitle

\begin{abstract} 
Motivated by an indecomposability criterion of Xun Lin for  the bounded derived category of coherent sheaves on a smooth projective variety $X$, we study the paracanonical base locus  of $X$, that is the intersection of the base loci of  $\omega_X \otimes P_{\alpha}$ for all $\alpha \in \Pic0 X$. 
We prove that this is equal to
 the relative base locus of $\omega_X$ with respect to  the Albanese morphism of $X$. 
As an application, we get that bounded
 derived categories of Hilbert schemes of points on certain surfaces
do not admit non-trivial semi-orthogonal decompositions. We also have a consequence on the indecomposability of bounded derived categories in families.
Finally, our viewpoint
 allows to  unify and extend   some results  
 recently appearing in the literature.  
\end{abstract}

\section{Introduction}

In this note, we are interested in finding hypothesis ensuring the \emph{indecomposability} of the bounded derived category of a variety $X$, that is the nonexistence of non-trivial semi-orthogonal decompositions of $\DD^b(X)$ (see \S 2.1 for  definitions).
In recent years,  bounded derived categories of projective varieties -- and their semi-orthogonal decompositions -- have been much studied.  
The interest in their indecomposability 
 mainly rests on its conjectural relation with the minimal model program (see \cite{kaok} and \S 2.1):
\begin{conjecture}\label{conj1}
Let $X$ be a smooth projective variety.
If $\DD^b(X)$ has no non-trivial semi-orthogonal decompositions, then $X$ is minimal, i.e., the canonical line bundle $\omega_X$ is nef.
\end{conjecture}

 Examples of varieties whose derived categories admit no non-trivial semiorthogonal decompositions are,  for instance,
varieties with trivial, or, more generally,  algebraically trivial canonical bundle (\cite{br} and \cite[Corollary 1.7]{kaok}, respectively), 
curves of genus $\geq 1$ \cite{ok}, 
varieties whose Albanese morphism is finite \cite[Theorem 1.4]{pi}.\footnote{To be precise, in \cite[Theorem 1.4]{pi}, Pirozhkov proved that such varieties are  noncommutatively stably semiorthogonally indecomposable, which is a stronger notion than indecomposability. 
}

It is well known that the converse  direction  in Conjecture \ref{conj1} is false (a counterexample is furnished by Enriques surfaces \cite{zu}), but
we have the following folklore  (see \cite[Conjecture 1.6]{bigole}, or \cite[Question E]{baetal}):

\begin{conjecture}\label{conjsuff}
Let $X$ be a smooth projective variety. If  $\omega_X$ is nef and effective, then $\DD^b(X)$ has no non-trivial semi-orthogonal decompositions.
\end{conjecture}
\noindent At the moment of writing, this conjecture (as even the above Conjecture \ref{conj1}) is widely open in general and there are just some (classes of) varieties for which it has been 
verified: see \cite{kaok, baetal}, besides the references already quoted above.
A particularly interesting case is given by  symmetric products of curves.
It has been conjectured by Belmans-Galkin-Mukhopadhyay \cite[Conjecture 1.1]{begamu} and, independently, by  Biswas-G\'omez-Lee \cite[Conjecture 1.4]{bigole}, that the bounded derived category of the $n$-\emph{th} symmetric product of a smooth projective curve of genus $g \geq 2$ has no non-trivial semi-orthogonal decompostions, if $n \leq g - 1$. 
Very recently,  Lin \cite[Theorem 1.9]{li}  nicely proved  it by using  a criterion that we will recall below (see Theorem \ref{Linthm}). This complements some previous partial results  in  \cite{bigole} and \cite{baetal}.

A result of Kawatani-Okawa   is particularly inspiring for us. These authors established a relation between the base locus $\Bs|\omega_X|$ of the canonical line bundle of $X$ and the non-existence of
semi-orthogonal decompositions (\emph{SODs} for short) of $\DD^b(X)$.
 Among other things,  they proved \cite[Corollary 1.5]{kaok}:
\begin{theorem}[Kawatani-Okawa]\label{KOthm}
If $\Bs|\omega_X|$ is a finite set (possibly empty), then $\DD^b(X)$ has no non-trivial SODs.
\end{theorem}  

\noindent So, in particular, a weak version of Conjecture \ref{conjsuff} is known to be true:
\emph{the global generation of $\omega_X$ implies the nonexistence of SODs of $\DD^b(X)$}.
In \cite{li},  Lin   focused instead on the \emph{paracanonical base locus}  of a variety $X$, that is the closed subset
\[
\PBs|\omega_X| := \bigcap_{\alpha \in \Pic0X} \Bs|\omega_X \otimes P_{\alpha}|,
\]
where $\omega_X$ is the canonical line bundle of $X$, and   $P_{\alpha}$ denotes the topologically trivial line bundle on $X$ corresponding to the (closed) point $\alpha$ of the Picard variety $\Pic0X$.
 Lin refined Kawatani-Okawa theorem as follows:
\begin{theorem}[Lin]\label{Linthm}
If $\PBs|\omega_X|$ is a finite set (possibly empty), then $\DD^b(X)$ has no non-trivial SODs.
\end{theorem}

 The study of paracanonical base loci  has its own interest. For instance, their emptiness (and related properties)  has been implicitly studied in, e.g., \cite{PP3, barja, eventual}, where several results of quite different flavor are obtained (see also \cite{paposurvey} and the references therein).  
Note that, since $\omega_X$ is an invertible sheaf, by definition $\PBs|\omega_X|$ equals the support of the cokernel of the sum of  evaluation maps
\[
\bigoplus_{\alpha \, \in \, \Pic0 X} H^0(X, \omega_X \otimes P_{\alpha}) \otimes P_{\alpha}^{\vee} \rightarrow \omega_X.
\]
Here we observe that, using some standard results of generic vanishing theory due to Chen, Jiang, Pareschi, Popa and Schnell (that will be recalled in \S 2.2 for reader's convenience), it is possible, quite easily, to give a different account of this locus, introducing a ``relative'' point of view (see  Theorem \ref{mainthm} below). This allows to generalize Theorem \ref{KOthm} to a relative setting where the Albanese morphism 
\[
a_X \colon X \rightarrow \Alb X
\]
 of $X$ appears. More precisely,    Lin's criterion  can be read  as
a relative version of  Kawatani and Okawa's one (see Corollary \ref{maincor} below).
 Before giving the statement, let us fix some other notations. 
We denote as usual by
\[
a_X^*{a_X}_*\omega_X \rightarrow \omega_X
\]
 the adjuction morphism.  The corresponding relative base ideal is 
\[
\mathfrak{b}_X := \mathfrak{b}(\omega_X, a_X) = \mathrm{Im}\bigl[ a_X^*{a_X}_*\omega_X \otimes \omega_X^{\vee} \rightarrow \OO_X \bigr],
\]
and the \emph{relative base locus of $\omega_X$ with respect to the Albanese morphism} $a_X$ is, by definition, the closed subset cut out by
 the relative base ideal $\mathfrak{b}_X$.

Then, our main  result is

\begin{theorem}\label{mainthm}
Let $X$ be a smooth projective variety, or a compact K\"ahler manifold.\footnote{See the comment below Theorem \ref{cjthm}.} Then the paracanonical base locus  $\PBs|\omega_X|$  equals the relative base locus of $\omega_X$ with respect to the Albanese morphism of $X$.
\end{theorem}
\noindent In particular, we get that
 \emph{$\omega_X$ is $a_X$-relatively globally generated} (this means that the natural morphism $a_X^*{a_X}_*\omega_X \rightarrow \omega_X$ is surjective, i.e., $\mathfrak{b}_X = \OO_X$) \emph{if and only if  $\PBs|\omega_X|$  is empty}.

Since
$a_X$ finite implies  $a_X^* {a_X}_*\omega_X \twoheadrightarrow \omega_X$, we have that
\begin{corollary}\label{corfin}
The paracanonical base locus of 
a variety  with finite Albanese morphism is empty. 
\end{corollary}

\begin{remark}
Similarly, if $X$ has maximal Albanese dimension (i.e., the Albanese morphism $a_X$ is generically finite onto its image),  the paracanonical base locus of $X$ is contained in the exceptional locus of $a_X$, which is defined as the inverse image, via $a_X$, of the points in $a_X(X)$ having non-finite fibers. 
This gives a partial answer to \cite[Question 8.7]{mendes}.
\end{remark}

Moreover, from Theorems \ref{mainthm} and \ref{Linthm}, it follows
\begin{corollary}\label{maincor}
If $\mathfrak{b}_X$ defines a finite set (possibly empty), then $\DD^b(X)$ has no non-trivial SODs. 
\end{corollary}

\noindent This partially generalizes the above-mentioned  result of Pirozhkov on varieties having finite Albanese morphism. 
 It also allows to give an alternative proof of
\cite[Theorem 1.9]{li} on indecomposability of derived categories of symmetric products of curves (see \S 4.2), and to put both these results under the same perspective.

On the other hand, Okawa recently informed us that  
he can prove  the indecomposability of the derived category of a smooth projective surface $S$ with nef $\omega_S$ and irregularity $h^0(S, \Omega_S^1) = \dim \Alb S > 0$.
It would be interesting to know if the same holds true in higher dimensions, for instance in the case of minimal  varieties of maximal Albanese dimension.

Our second result concerns indecomposability of derived categories of punctual
Hilbert scheme on surfaces.
Let $S$ be a smooth projective surface and $S^{[n]}$  the Hilbert scheme of points of length $n$ on $S$. In \cite[Proposition 5.1]{baetal}, the authors proved that, if $\Bs|\omega_S|$ is empty, then, for all $n \geq 1$, $\Bs|\omega_{S^{[n]}}|$ is empty too. Hence, by Theorem \ref{KOthm}, $\DD^b(S^{[n]})$ has no non-trivial SODs. 
It is natural to ask if something similar holds true  taking   into account  the paracanonical base locus  instead of the usual canonical base locus.
Theorem \ref{mainthm} turns out to be useful in this case.
Indeed, by using it, we prove 
\begin{theorem}\label{hilbthm}
Let $n > 0$ be a natural number.
 If $$\bigcap_{\alpha \in U} \Bs|\omega_S \otimes P_{\alpha}| = \varnothing $$
for all non-empty open subset $U \subseteq \Pic0 S$,\footnote{This means that $\omega_S$ is \emph{continuously globally generated} with respect to the Albanese morphism $a_S$ (see \S \ref{2.2} below). This notion, in general, neither implies nor it is implied by the global generation of $\omega_S$.}   
then $\PBs|\omega_{S^{[n]}}| = \varnothing$. 
Therefore,
$\DD^b(S^{[n]})$ has no non-trivial SODs.
\end{theorem}

\noindent 
This should be seen as a counterpart to the symmetric product of curves situation mentioned above (see also \S 4.2 and Remark \ref{ultimormk}).
It  adds a new class of examples 
to the list of varieties appearing in the literature, whose derived categories are indecomposable. For instance,
 it applies to surfaces having finite Albanese morphism and Albanese image of general type thanks to \cite[Proposition 5.5(iii)]{PP3} (see also \cite[Lemma 2.1]{jlt}).\footnote{Recall that such properties are inherited by taking finite ramified coverings.}

The proof of Theorem \ref{hilbthm} will be given in $\S 4.3$. The last statement follows from Lin's criterion (Theorem \ref{Linthm}).

Lastly, we consider indecomposability in families and we give  a variant of \cite[Corollary B]{baetal}, where the authors proved that, if $f \colon \mathcal{X} \rightarrow T$ is a smooth projective family of varieties, with $T$ an irreducible algebraic variety over $\mathbb{C}$, and  there exists a point $t_0 \in T$ such that $\Bs|\omega_{\cX_{t_0}}|$ is finite (possibly empty), then $\DD^b(\cX_t)$ has no non-trivial SODs for all $t \in T$. Here, as usual,  $\cX_t$ denotes the variety $f^{-1}(t)$.
By assuming that  the Albanese morphisms  of the  varieties appearing in the family $f$ vary ``smoothly'',
 we can prove that the same result is true, if the paracanonical base locus $\PBs|\omega_{\cX_{t_0}}|$ is finite (possibly empty) for a point $t_0 \in T$.  
 Namely,
in order to precisely  state our result, let us recall that, given a  smooth projective family $f \colon \cX \rightarrow T$, there exists
 a commutative diagram
\begin{equation}\label{notation0}
\xymatrix{
\cX \ar[r]^-{\varphi} \ar[d]_f &\Alb \cX / T  \ar[ld]^{\eta}\\
T
}
\end{equation}
where $\eta$ is a smooth morphism  such that, for all $t \in T$, the fiber $(\Alb \cX / T)_t$ is $\Alb \mathcal{X}_t$, and
 the restriction $\varphi_t \colon \cX_t \rightarrow (\Alb \cX / T)_t$ is the Albanese morphism $a_{\cX_t} \colon \cX_t \rightarrow \Alb \cX_t$ of $\cX_t$. The morphism $\varphi$  is a universal morphism from the family $f$ to families of abelian varieties over $T^{\prime}$, for any morphism $T^{\prime} \rightarrow T$ (see \cite[n.\ 236]{gr2}, or \cite{fu, campana} and the references therein):  
 $\varphi$ is called the \emph{relative Albanese morphism} for $f$ and $\Alb \mathcal{X} / T$ is the \emph{relative Albanese variety} for $f$.
Thanks to Theorem \ref{mainthm}, we have the following result whose proof is the content of \S 5:
\begin{theorem}\label{family}
Let $f \colon \mathcal{X} \rightarrow T$ be a smooth projective family of varieties as above, and assume that its relative Albanese morphism $\varphi$ is  smooth. If  $\PBs|\omega_{\cX_{t_0}}|$ is finite (possibly empty) for a certain point $t_0 \in T$, then $\DD^b(\cX_t)$ has no non-trivial SODs for all $t \in T$.
\end{theorem}

We believe (and hope) that Theorem \ref{mainthm} could also be useful in some other context, possibly far from those considered in this note.

\vskip0.3truecm\noindent\textbf{Notations.} 
All varieties  are assumed to be smooth, projective, irreducible and defined over $\mathbb C$.
A coherent sheaf on a variety is simply called a \emph{sheaf}.
Let $\mathcal{P}$ be the 
 normalized Poincar\'e line bundle on $\Alb X \times \Pic0 X$. Given a closed point $\alpha \in \Pic0 X$, we denote by $P_{\alpha} := \mathcal{P}|_{\Alb X \times \{\alpha\}}$ the corresponding line bundle on $\Alb X$. Recall that $\Pic0(\Alb X) \cong \Pic0 X$ via the pullback along the  Albanese morphism $a_X \colon X \rightarrow \Alb X$ of $X$. So, by a slight abuse of notation,   the line bundle $a_X^*P_{\alpha}$ on $X$ corresponding to $\alpha \in \Pic0 X$ is simply denoted by
 $P_{\alpha}$.

\vskip0.3truecm\noindent\textbf{Acknowledgements.}   
I would like to thank  Beppe Pareschi  and Paolo Stellari. I'm  especially grateful  to Xun Lin for  very valuable comments  
and useful communications,
and to  \'Angel D.\ R.\ Ortiz for our always fruitful conversations.
Finally, I  thank the anonymous referee, whose meticulous reading and valuable criticism have been of great help to me.

\section{Prerequisites 
}

Here we recall the definition of semi-orthogonal decomposition of a derived category, and some notions and results concerning the generic vanishing theory.

\subsection{Semi-orthogonal decompositions}

Let $\DD^b (X)$ be the bounded derived category of a smooth projective variety $X$. 
 If there exist full triangulated subcategories $\mathcal{A}$ and $\mathcal{B}$ of $\DD^b(X)$ such that 
\begin{equation}\label{sod}
\DD^b(X) = \left\langle \, \mathcal{A}, \, \mathcal{B} \, \right\rangle, 
\end{equation}
i.e., $\mathcal{A}$ and $\mathcal{B}$ generate the triangulated category $\DD^b(X)$, and
\[
\mathrm{Hom}(b, a) = 0
\]
for any object $a \in \mathcal{A}$ and $b \in \mathcal{B}$,
then we  say that \eqref{sod} is a  \emph{semi-orthogonal decomposition} (\emph{SOD} for short) of  $\DD^b(X)$.
\begin{definition}
A semi-orthogonal decomposition  of $\DD^b(X)$ is said to be \emph{non-trivial} if $\mathcal{A}$ and $\mathcal{B}$ are both non-zero. 
If $\DD^b(X)$ has no non-trivial semi-orthogonal decompositions, it is said to be \emph{indecomposable}.
\end{definition}
We refer the reader to \cite{kusurvey} for an overview on SODs in algebraic geometry.
The indecomposability of the derived category of a variety is (conjecturally) strongly related to its birational geometry. 
Indeed, the \emph{$DK$-hypothesis} of \cite{ka} (see also \cite{kasurvey}) predicts that, if there exists a $K$-inequality between two smooth projective varieties $X$ and $Y$, say $X \leq_K Y$,\footnote{This means that there exists a third
 smooth projective variety $Z$ with birational
morphisms $f \colon Z \rightarrow X$ and $g \colon Z \rightarrow Y$ such that $g^*\omega_Y \otimes f^* \omega_X^{\vee}$ is effective.} then there is a fully faithful functor of triangulated categories $\DD^b(X) \hookrightarrow \DD^b(Y)$.
This would imply the existence of a semi-orthogonal decomposition of $\DD^b(Y)$ 
\[
\DD^b(Y) = \left\langle  \, \DD^b(X)^{\bot},\, \DD^b(X) \, \right\rangle,
\]
where $\DD^b(X)^{\bot}$ is the right orthogonal complement of $\DD^b(X)$ inside $\DD^b(Y)$. See \cite{kasurvey}, and the references therein, for more information about this (and related) ideas.

\subsection{Generic vanishing}\label{2.2}

Let $A$ be an abelian variety, and $\F$ be a sheaf on $A$. We denote by
\[
V^i(A, \F) := \{ \alpha \in \Pic0 A \ |\ h^i(A, \F \otimes P_{\alpha}) > 0 \}
\] 
the $i$-\emph{th cohomological support loci} of $\F$. These are closed algebraic subset of $\Pic0 A$ by semicontinuity.

\begin{definition}
The sheaf $\F$ is said to be: 

\noindent \emph{i)} $GV$, if $\mathrm{codim}_{\Pic0 A} V^i(A, \F) \geq i$, for all $i \geq 0$;

\noindent \emph{ii)} $M$-\emph{regular},  if $\mathrm{codim}_{\Pic0 A} V^i(A, \F) > i$, for all $i > 0$.
\end{definition}

\noindent We refer the reader to the surveys \cite{msri} and \cite{paposurvey} (and the references therein) for the main results surrounding
 such notions.

For us, the central property is the following: 
\begin{proposition}[\cite{PP1}, Proposition 2.13]
 Let $\F$ be an $M$-regular sheaf on $A$. Then, for any non-empty open subset $U \subseteq \Pic0 A$, the sum of evaluation maps
\begin{equation}\label{cgg}
\bigoplus_{\alpha \, \in \, U} H^0(A, \F \otimes P_{\alpha}) \otimes P_{\alpha}^{\vee} \rightarrow \F
\end{equation}
is surjective.
\end{proposition}
\noindent A sheaf satisfying the condition \eqref{cgg} is said to be \emph{continuously globally generated}. More generally, a sheaf $\G$ on a projective variety $Y$ admitting  a non-trivial morphism $g \colon Y \to B$ to an abelian variety $B$, is said to be \emph{continuously globally generated with respect to} $g$, if 
\[
\bigoplus_{\beta \, \in \, V} H^0(Y, \G \otimes g^*P_{\beta}) \otimes g^*P_{\beta}^{\vee} \rightarrow \G
\]
is surjective for all non-empty open subset $V \subseteq \Pic0 B$. 
This notion was introduced by Pareschi and Popa in \cite{PP1}, and
it has to be considered as a ``continuous'' variant of the usual notion of generation by global sections.

Now, let $a \colon X \rightarrow A$ be a morphism from a smooth projective variety $X$ to an abelian variety $A$.  The higher direct images  $R^j a_* \omega_X$ give non-trivial examples of  $GV$ sheaves (see \cite{ha}). Recently, it has been observed that indeed much more is true:

\begin{theorem}[Chen-Jiang decomposition]\label{cjthm}
Given a morphism $a \colon X \rightarrow A$ and an integer $j \geq 0$, there exists a canonical decomposition
\begin{equation}\label{cjdec}
R^j a_* \omega_X = \bigoplus_i \pi_i^*\F_i \otimes P_{\alpha_i},
\end{equation}
where $\pi_i \colon A \rightarrow A_i$ are quotient of abelian varieties with connected fibers, the $\F_i$'s are $M$-regular sheaves on $A_i$, and $\alpha_i \in \Pic0 A$ are torsion points.
\end{theorem}

\noindent This powerful result was proved by Chen-Jiang \cite{cj}, assuming $a$ generically finite, and soon after by Pareschi-Popa-Schnell \cite{paposc} in complete generality and even in the K\"ahler setting. More precisely, in \cite[Theorem A]{paposc}, the authors proved that the Chen-Jiang decomposition \eqref{cjdec} holds true for a  holomorphic mapping $a \colon X \rightarrow T$  from a compact K\"ahler manifold $X$ to a compact complex torus $T$, and, moreover, the $M$-regular sheaves $\F_i$, that are defined on the corresponding compact complex tori $T_i$,  have \emph{projective} support  for all $i$. This allows to say that the $\F_i$'s are continuously globally generated (see \emph{op.cit.}, \S 20), that is precisely what we need here. Indeed,  this fact -- along with the decomposition \eqref{cjdec} itself-- is more than sufficient to prove  Theorem \ref{mainthm}.

See also \cite{vi} for a simplified proof of Theorem \ref{cjthm} in the projective case, following the lines of \cite{cj}.
The Chen-Jiang decomposition  also holds  true for the pushforwards of pluricanonical bundles $a_* \omega_X^{\otimes m}$, $m\geq 2$, of a smooth projective variety $X$ by Lombardi-Popa-Schnell \cite{loposc}, and it was recently applied to the study of derived invariants in \cite{cp} and \cite{clp}.

\section{Proof of Theorem \ref{mainthm}}

We are now ready to give the proof of our first result.
We may assume that $X$ is smooth projective,
 given that the proof goes completely analogous in the compact K\"ahler case (see the comment below Theorem \ref{cjthm}).    
Let 
$a_X : X \rightarrow \Alb X$ 
be the Albanese morphism of $X$, and
\begin{equation}\label{CJ1}
{a_X}_*\omega_X = \bigoplus_i \pi_i^* \mathcal{F}_i \otimes P_{\alpha_i}
\end{equation}
be the Chen-Jiang decomposition of ${a_X}_* \omega_X$ (see \S 2.2). Hence, for each index $i$, $\pi_i \colon \Alb X \rightarrow A_i$ is a quotient of abelian varieties with connected fibers, $\F_i$ is an $M$-regular sheaf on $A_i$, and $\alpha_i \in \Pic0 X$ is a  torsion point. The sheaves $\F_i$ are continously globally generated by \cite[Proposition 2.13]{PP1}. This means, in particular, that  the sum of evaluation maps
\begin{equation}\label{eqproof1}
\bigoplus_{\alpha \, \in \, \Pic0 A_i} H^0(A_i, \F_i \otimes P_\alpha) \otimes P_\alpha^{\vee} \rightarrow \F_i
\end{equation}
is surjective. Note that the dual morphism $\widehat{\pi_i} \colon \Pic0 A_i \hookrightarrow \Pic0 X$ is injective, hence we simply denote by $\alpha = \widehat{\pi_i}(\alpha) \in \Pic0 X$ the image of an element $\alpha \in \Pic0A_i$.  Moreover, by the projection formula,
\begin{equation*}\label{pf}
H^0(\Alb X, \pi_i^*\F_i \otimes P_\alpha) =
H^0(A_i, \F_i \otimes P_\alpha)
\end{equation*}
if $\alpha \in \Pic0 A_i$.
 Hence,  applying $\pi_i^*$ to \eqref{eqproof1} and tensoring by $P_{\alpha_i}$, 
we get that, a fortiori, the morphism
\[
\bigoplus_{\alpha \, \in \, \Pic0 X} H^0(\Alb X, \pi_i^*\F_i \otimes P_\alpha) \otimes P_\alpha^{\vee} \otimes P_{\alpha_i} \rightarrow 
\pi_i^*\F_i \otimes P_{\alpha_i}
\]
is surjective.  
 Now we change notation as follows: we call $\alpha_i - \alpha =: - \beta  \in \Pic0 X$. So we have that the map 
\begin{equation}\label{eqproof3}
\bigoplus_{\beta \, \in \,  \Pic0 X} H^0(\Alb X, \pi_i^*\F_i \otimes P_{\alpha_i} \otimes P_{\beta} ) \otimes P_\beta^{\vee} \rightarrow \pi_i^*\F_i \otimes P_{\alpha_i}
\end{equation}
is surjective, and this holds true for any index $i$ appearing in the Chen-Jiang decomposition of ${a_X}_*\omega_X$. Therefore, taking the sum over $i$ in \eqref{eqproof3} -- and using \eqref{CJ1} --, we get the following
\begin{lemma}\label{lemma1}
Let $X$ be a smooth projective variety, or a compact K\"ahler manifold. Then, the sum of evaluation maps
\begin{equation}\label{eqproof2}
\bigoplus_{\beta \, \in \, 
 \Pic0 X} H^0(\Alb X, {a_X}_*\omega_X \otimes P_{\beta} ) \otimes P_\beta^{\vee} \rightarrow {a_X}_*\omega_X
\end{equation}
is surjective.
\end{lemma}
Now, if  we apply $a_X^*$ to \eqref{eqproof2} and post-compose with the natural morphism $a_X^* {a_X}_* \omega_X \rightarrow \omega_X$, we recover the evaluation morphism
\begin{equation}\label{eqfinal0}
\xymatrix{\bigoplus_{\beta \, \in \, \Pic0 X} H^0(X, \omega_X \otimes P_{\beta} ) \otimes P_\beta^{\vee} \ar@/^1.5pc/[rr]^{} \ar@{->>}[r] &a_X^* {a_X}_*\omega_X \ar[r] &\omega_X,
}
\end{equation}
where we used that $H^0(\Alb X, {a_X}_*\omega_X \otimes P_{\beta} ) = H^0(X, \omega_X \otimes P_{\beta} )$ for all $\beta \in \Pic0 X$. 
The proof of Theorem \ref{mainthm}  follows from \eqref{eqfinal0}. Indeed, if $x \in X$ is not contained in the relative base locus of $\omega_X$ with respect to $a_X$, then the sum of evaluation maps
\[
\bigoplus_{\beta \, \in \,
\Pic0 X} H^0(X, \omega_X \otimes P_{\beta} ) \otimes P_\beta^{\vee} \rightarrow \omega_X,
\]
is surjective at $x$ by \eqref{eqfinal0}. Therefore, by definition, $x \notin \PBs|\omega_X|$. 
The other inclusion is completely similar (and indeed easier, since it does not need Lemma \ref{lemma1}).
\qed

\begin{remark}\label{rmkgenerale}
Let $a \colon X \rightarrow A$ be a morphism from a smooth projective variety $X$ to an abelian variety $A$. The properties of ${a_X}_* \omega_X$ we used in the proof of Theorem \ref{mainthm},  hold also true for $a_* \omega_X$ (see \S 2.2). Therefore, if we define
the \emph{$a$-paracanonical base locus}  as 
\[
\PBs_{a}|\omega_X| := \bigcap_{\alpha\, \in \, \Pic0 A} \Bs|\omega_X \otimes a^*P_{\alpha}|,
\]
we have that this is equal to the relative base locus of $\omega_X$ with respect to the morphism $a \colon X \rightarrow A$. Note that $\PBs|\omega_X| \subseteq \PBs_a |\omega_X|$, for any $a$. 

More generally, the same happens for a holomorphic map $a : X \rightarrow T$ from a compact K\"ahler manifold $X$ to a compact complex
torus $T$, and, in the projective setting,  
 for $a_*\omega_X^{\otimes m}$ with $m\geq 2$ 
(see the comment below Theorem \ref{cjthm}), once one defines  the \emph{$a$-parapluricanonical base locus} $\PBs_{a}|\omega_X^{\otimes m}|$ likewise.
\end{remark}

\section{Examples}

In this section, we present some non-trivial examples where the hypothesis of Corollary \ref{maincor} can be (quickly) checked. 

\subsection{Varieties with finite Albanese morphism} As already observed, it is well known  that, if a variety $X$ has finite Albanese morphism $a_X$, then its canonical bundle $\omega_X$ is $a_X$-relatively globally generated. 
Thus, $\DD^b(X)$ is indecomposable if $a_X$ is finite.
We refer the reader to the paper \cite{kaabelian} for some fundamental structural results on varieties with finite Albanese morphism, and to \cite{pi} for a stronger theorem on indecomposability, in this case.

\subsection{Symmetric products of curves}
Let $C$ be a smooth projective curve of genus $g \geq 2$. Let $J(C)$ be its Jacobian variety. 
We denote by $C^{(n)}$  the $n$-\emph{th} symmetric product of  $C$. 
Let
\begin{equation}\label{albsym}
a_n := a_{C^{(n)}} : C^{(n)} \rightarrow J(C), \quad D \mapsto \OO_C(D- np_0)
\end{equation}
be the Albanese morphism of $C^{(n)}$, where $p_0 \in C$ is a fixed point.

In \cite{li}, by algebraically moving techniques,  Lin proved that $\PBs|\omega_{C^{(n)}}|$ is empty if $n \leq g-1$. In particular, $\DD^b(C^{(n)})$ has no non-trivial SODs in that range. 
By making use of Theorem \ref{mainthm}, 
we  see how this fact also follows from a classical description of the symmetric product of curves appearing in \cite{sch}.\footnote{Interestingly enough,  the main result of \cite{sch} was also used by Belmans and Krug (see Remark \ref{ultimormk} below) to give a simpler proof of a result of Toda constructing a non-trivial semi-orthogonal decomposition of $\DD^b(C^{(n)})$, when $n \geq g \geq 2$.}
Namely, we have
\begin{proposition}\label{propcur}
Let $n \leq g-1$. Then the canonical bundle $\omega_{C^{(n)}}$ is $a_n$-relatively globally generated.
\end{proposition}
\begin{proof}
Let 
\[
i_{n-1} : C^{(n-1)} \hookrightarrow C^{(n)}, \quad D \mapsto D + p_0,
\]
and  denote its image  by $x = i_{n-1}(C^{(n-1)}) \subseteq C^{(n)}$.
Let $\Theta$ be the theta divisor on $J(C)$, and 
  $\theta \in \mathrm{NS}(C^{(n)})$  the class of the pullback of $\Theta$ via $a_n$.
We recall that algebraic and numerical
equivalence of divisors coincide on $C^{(n)}$, and that
\[
\omega_{C^{(n)}} = \theta + (g-n-1) x 
\]
holds in $\mathrm{NS}(C^{(n)})$ (see, e.g., \cite[Lemma 2.1]{bigole}).
We prove that, for any $m \geq 0$, the map
\[
a_n^* {a_n}_* \OO_{C^{(n)}}(\theta + m x) \to \OO_{C^{(n)}}(\theta + m x) 
\]
is surjective. The case $m = g - n - 1$ is the desired result. Since $\theta$  is a pullback line bundle via $a_n$, it is enough to prove this without the twist by $\theta$. For $m=0$ this is evident, while for $m=1$ this holds by
 \cite[Proposition 10]{sch}.
Assume $m \geq 2$. Then, we have that $a_n^* {a_n}_* \OO_{C^{(n)}}(mx) \rightarrow \OO_{C^{(n)}}(mx)$ is surjective by induction on $m$, thanks to the following commutative diagram
\[
\xymatrix{
a_n^* ({a_n}_* \OO_{C^{(n)}}((m-1)x) \otimes {a_n}_* \OO_{C^{(n)}}(x))  \ar[r]^-{\mathrm{ev} \otimes \mathrm{ev}} \ar[d] &\OO_{C^{(n)}}((m-1)x) \otimes \OO_{C^{(n)}}(x) = \OO_{C^{(n)}}(mx) \ar@{=}[d]\\
a_n^* ({a_n}_* \OO_{C^{(n)}}(mx)) \ar[r]^{\mathrm{ev}} &\OO_{C^{(n)}}(mx).}
\] 
\end{proof}

\subsection{Hilbert schemes of points on surfaces} Here we prove Theorem \ref{hilbthm}.

Let $S$ be a smooth projective surface. Given an integer $n >0$, we denote by $S^{[n]}$ the Hilbert scheme of points of length $n$ on $S$, and by $S^{(n)}$ the $n$-fold symmetric product. It is well-known (see \cite{fo}) that $S^{[n]}$ is smooth and that the Hilbert-Chow morphism $\pi \colon S^{[n]} \rightarrow S^{(n)}$ is a canonical resolution of singularities of $S^{(n)}$.
Moreover,
\begin{equation}\label{albhilb}
a_{S^{[n]}} : S^{[n]} \xrightarrow{\pi} S^{(n)} \xrightarrow{f_n} \Alb S
\end{equation}
is the Albanese morphims of $S^{[n]}$, where $f_n$ is the morphism induced by $a_S$ by addition on $\Alb S$. In particular, 
$\Alb (S^{[n]}) \cong \Alb S$ (see \cite{fo2}).\footnote{We also refer to \cite{hui1, hui2} for a detailed account on these morphisms.}
Let us point out that  
\eqref{albhilb} is a natural generalization of the morphism \eqref{albsym}.

The Hilbert-Chow morphism $\pi$ is a crepant resolution, i.e.,
$\omega_{S^{[n]}} = \pi^*\omega_{S^{(n)}}$, and $S^{(n)}$ is a Gorenstein variety, so that $\omega_{S^{(n)}}$ is a line bundle.
It is easy to see that \emph{$\omega_{S^{[n]}}$ is $a_{S^{[n]}}$-relatively globally generated, if $\omega_{S^{(n)}}$ is $f_{n}$-relatively globally generated}. Indeed, since $\pi_*\OO_{S^{[n]}} = \OO_{S^{(n)}}$, by the projection formula we have that 
\[
a_{S^{[n]}}^* {a_{S^{[n]}}}_* \omega_{S^{[n]}} = \pi^*f_n^*{f_n}_*\pi_* \pi^* \omega_{S^{(n)}} = \pi^*f_n^*{f_n}_*\omega_{S^{(n)}}.
\]
Hence, if 
\[
f_n^*{f_n}_*\omega_{S^{(n)}} \rightarrow \omega_{S^{(n)}}
\]
is surjective,  by applying $\pi^*$ to it we get that
\[
a_{S^{[n]}}^* {a_{S^{[n]}}}_* \omega_{S^{[n]}} \rightarrow \omega_{S^{[n]}}
\]
is surjective, too.

So we reduced to prove 
\begin{proposition}\label{propreduction}
If $\omega_S$ is continuously globally generated with respect to $a_S$, then 
\[
f_n^* {f_n}_* \omega_{S^{(n)}} \rightarrow \omega_{S^{(n)}}
\]
is surjective.
\end{proposition}
\begin{proof}
Let $S^n$  be the usual product variety. We consider the following commutative diagram
\begin{equation}\label{quotient}
\xymatrix{
S^n  \ar[r]^-{q} \ar[d]^-{a_{S^n}} &S^{(n)}  \ar[d]^-{f_n}  \\
(\Alb S)^n  \ar[r]^-{\sigma} & \Alb S}
\end{equation}
where $q$  is the quotient morphism by the action of the symmetric group $\mathfrak{S}_n$ on $S^n$, $a_{S^n} = (a_S)^n$ is  the Albanese morphism  of $S^n$, and $\sigma$ is the addition morphism.

We make use of the following technical  fact:
\begin{proposition}\label{generalgg}
Let $X$ be a smooth projective variety such that
 $$\bigcap_{\alpha \in U} \Bs|\omega_X \otimes P_{\alpha}| = \varnothing,$$
for any non-empty open subset $U \subseteq \Pic0 X$.
Let $n \geq 1$ be a natural number.
If $\sigma_n \colon \Alb X ^n \rightarrow \Alb X$ denotes the addition morphism,
then
the sum of  evaluation maps
\[
ev_{n,\, U} : \bigoplus_{\alpha \, \in \, U} H^0(X^n, \omega_{X^n} \otimes \sigma_n^* P_{\alpha}) \otimes \sigma_n^* P_{\alpha}^{\vee} \rightarrow \omega_{X^n} 
\]
is surjective, for all non-empty open subset $U \subseteq \Pic0 X$.\footnote{Otherwise said: 
if $\omega_X$ is continuously globally generated with respect to $a_X$,
 then $\omega_{X^n}$ is continuously globally generated with respect to the composition $a_n := \sigma_n \circ a_{X^n} \colon X^n \rightarrow \Alb X$.}
\end{proposition}

\noindent Granting it for the moment, we apply this to the surface $X = S$.
 So, in particular, 
$$\bigcap_{\alpha \, \in \, \Pic0 S }\Bs|\omega_{S^n} \otimes \sigma^* P_{\alpha}| = \varnothing.$$  
Since $S^n$ is quasi-compact, there exists a finite number $N$ such that, for some $\alpha_1, \ldots, \alpha_N \in \Pic0 S$, the intersection $\bigcap_{k=1}^N \Bs|\omega_{S^n} \otimes \sigma^* P_{\alpha_k}|$ is empty, i.e., 
the sum of evaluation  maps
\begin{equation}\label{eqfinal}
\bigoplus_{k=1}^N H^0(S^n, \omega_{S^n} \otimes \sigma^* P_{\alpha_k} ) \otimes \sigma^* P_{\alpha_k}^{\vee} \rightarrow \omega_{S^n}
\end{equation}
is surjective. 
Looking at the diagram \eqref{quotient}, note that  
$q_* \sigma^* P_{\alpha_k}^{\vee} = q_* a_{S^n}^* \sigma^* P_{\alpha_k}^{\vee} = q_* q^* f_n^* P_{\alpha_k}^{\vee} = f_n^* P_{\alpha_k}^{\vee} \otimes q_* \OO_{S^n}$
for all $k$, by the projection formula. Moreover, since $\omega_{S^n} = q^*\omega_{S^{(n)}}$, we have that $q_* \omega_{S^n} = \omega_{S^{(n)}} \otimes q_*\OO_{S^n}$ as well.
Since $q$ is finite,  we see,  by
taking the  pushforward $q_*$ in \eqref{eqfinal}, that the morphism
\begin{equation*}\label{surj0}
\bigoplus_{k=1 }^N H^0(S^n, \omega_{S^n} \otimes \sigma^* P_{\alpha_k} ) \otimes f_n^* P_{\alpha_k}^{\vee} \otimes q_* \OO_{S^n} \rightarrow \omega_{S^{(n)}} \otimes q_*\OO_{S^n}
\end{equation*}
is surjective, too.  Now, since $\OO_{S^{(n)}}$ is contained in $q_* \OO_{S^n}$ as a direct summand, there exists a surjective morphism $q_* \OO_{S^n} \to \OO_{S^{(n)}}$ and we get the following commutative diagram
\[
\xymatrix{
\bigoplus_{k=1 }^N H^0(S^n, \omega_{S^n} \otimes \sigma^* P_{\alpha_k} ) \otimes f_n^* P_{\alpha_k}^{\vee} \otimes q_* \OO_{S^n} \ar@{->>}[r]
 \ar@{->>}[d] & \omega_{S^{(n)}} \otimes q_*\OO_{S^n} \ar@{->>}[d] \\
\bigoplus_{k=1 }^N H^0(S^n, \omega_{S^n} \otimes \sigma^* P_{\alpha_k} ) \otimes f_n^* P_{\alpha_k}^{\vee}  \ar[r]
  & \omega_{S^{(n)}}  
}
\]
So, the bottom horizontal map is surjective, and, 
 by adjunction,  it factorizes as
\begin{equation*}\label{fact5}
\bigoplus_{k=1 }^N H^0(S^n, \omega_{S^n} \otimes \sigma^* P_{\alpha_k} ) \otimes f_n^* P_{\alpha_k}^{\vee} \rightarrow 
f_n^* {f_n}_* \omega_{S^{(n)}} \rightarrow
 \omega_{S^{(n)}}.
\end{equation*}

In order to  conclude the proof of Proposition \ref{propreduction} -- and hence
of Theorem \ref{hilbthm} --
it only remains to give the 
\begin{proof}[Proof of Proposition \ref{generalgg}]
The proof is essentially the same  of \cite[Proposition 2.12]{PP1} and it uses the method of ``reducible sections''. 
We proceed by induction on $n$.
For $n=1$ there is nothing to prove.
Let us denote by $\pi_i \colon X^n \rightarrow X$ the $i$-\emph{th} projection.
Since $\alpha \in \Pic0 X$, we have 
\[
\sigma_n^* P_{\alpha} = \pi_1^*P_{\alpha} \otimes \ldots \otimes \pi_n^* P_{\alpha} =: P_{\alpha}^{\boxtimes n},
\]
where $\sigma_n \colon \Alb X^n \rightarrow \Alb X$ is the addition morphism.
This can be easily proved from the see-saw principle (see, e.g., \cite[pp.\ 74-75]{mu}) and induction on $n$. 
  
Let $\bar{x} = (x_1, \ldots, x_n) \in X^n$ be an arbitrary point. 
Since we are assuming that $\omega_X$ is continuously globally generated with respect to $a_X$,   
there  exists a non-empty open subset $U_{x_n} \subseteq \Pic0 X$ such that $x_n \notin \Bs|\omega_X \otimes P_{\alpha}|$ for \emph{all} $\alpha \in U_{x_n}$.\footnote{Proof: Let $\mathcal{I}_{x_n}$ be the ideal sheaf of $x_n$ in $X$.
By upper semicontinuity, there exist  open subsets $U_1, U_2 \subseteq \Pic0 X$ such that $h^0(X, \omega_X \otimes P_{\alpha})$ (resp.\ $h^0(X, \omega_X \otimes \mathcal{I}_{x_n} \otimes P_{\alpha})$) takes its minimal value, when $\alpha \in U_1$ (resp.\ $\alpha \in U_2$).  Let $U_{x_n} := U_1 \cap U_2$. Then, 
since by the hypothesis the morphism
\[
\bigoplus_{\alpha \in U_{x_n}} H^0(X, \omega_X \otimes P_{\alpha}) \otimes P_{\alpha}^{\vee} \to \omega_X
\] 
is surjective at $x_n$, we must have that
$h^0(X, \omega_X \otimes \mathcal{I}_{x_n} \otimes P_{\alpha}) < h^0(X, \omega_X \otimes P_{\alpha})$ for all $\alpha \in U_{x_n}$. This implies that the evaluation morphism $H^0(X, \omega_X \otimes P_{\alpha}) \otimes \OO_X \to \omega_X \otimes P_{\alpha}$ is surjective at $x_n$, for all $\alpha \in U_{x_n}$. 
}

Let $U \subseteq \Pic0 X$ be a non-empty open subset, and denote $V_{x_n} := U \cap U_{x_n}$.
Let $p \colon X^n =  X^{n-1} \times X \rightarrow X^{n-1}$ be  the first projection.
We consider the following commutative diagram where maps are given by evaluation:\footnote{Note that $H^0(p^*\omega_X^{\boxtimes n-1} \otimes p^* P_{\alpha}^{\boxtimes n-1}) \neq 0$ if and only if $H^0(\pi_n^*\omega_X \otimes \pi_n^* P_{\alpha}) \neq 0$.}
\[
\resizebox{165mm}{!}{
\xymatrix{
\bigoplus_{\alpha \in V_{x_n}} H^0(p^*\omega_{X}^{\boxtimes n-1} \otimes p^*P_{\alpha}^{\boxtimes n-1}) \otimes H^0(\pi_n^*\omega_X \otimes \pi_n^*P_{\alpha}) \otimes \sigma_n^* P_{\alpha}^{\vee}  \ar@{=}[r] \ar[d] & \bigoplus_{\alpha \in  V_{x_n}} H^0(\omega_{X^n} \otimes \sigma_n^*P_{\alpha}) \otimes \sigma_n^* P_{\alpha}^{\vee} \ar[d] \\
\bigoplus_{\alpha \in V_{x_n}} H^0(p^*\omega_{X}^{\boxtimes n-1} \otimes p^*P_{\alpha}^{\boxtimes n-1}) \otimes p^*(P_{\alpha}^{\vee})^{\boxtimes n-1} \otimes \pi_n^* \omega_X \otimes \mathbb{C}(\bar{x}) \ar[r] & \omega_{X^n} \otimes \mathbb{C}(\bar{x})}}
\]
By the inductive hypothesis,  the bottom horizontal morphism  is surjective.\footnote{Indeed,  the projection formula gives $$H^0(p^*\omega_{X}^{\boxtimes n-1} \otimes p^*P_{\alpha}^{\boxtimes n-1}) = H^0(p_*\OO_{X^n} \otimes \omega_X^{\boxtimes n-1} \otimes P_{\alpha}^{\boxtimes n-1}) = H^0(\omega_{X^{n-1}} \otimes \sigma_{n-1}^*P_{\alpha}).$$} 
Moreover, since $V_{x_n} \subseteq U_{x_n}$, the left vertical morphism is surjective as well. Therefore, by Nakayama's lemma, 
$$\bar{x} \notin \mathrm{Coker}(ev_{n,\, V_{x_n}}) = \bigcap_{\alpha \in V_{x_n}} \Bs|\omega_{X^n} \otimes \sigma_n^* P_{\alpha}|.$$ 
Since $V_{x_n} \subseteq U$, one has
 $\bigcap_{\alpha \in U} \Bs|\omega_{X^n} \otimes \sigma_n^* P_{\alpha}| \subseteq \bigcap_{\alpha \in V_{x_n}} \Bs|\omega_{X^n} \otimes \sigma_n^* P_{\alpha}|$, and, 
therefore,  $\bar{x} \notin \mathrm{Coker}(ev_{n,\, U})$. Since the point $\bar{x} \in X^n$ was arbitrary, we finally have 
 that $ev_{n,\,  U}$ is surjective.
\end{proof}
This concludes the proof of Theorem \ref{hilbthm}.
\end{proof}

\begin{remark}\label{ultimormk}
As observed, the case of surfaces treated in \S 4.3 is the natural generalization of the one of curves of \S 4.2.
However, comparing Proposition \ref{propcur} and Theorem \ref{hilbthm}, 
it appears a  visible difference   on the index $n$ in the two cases.
The reason behind that
has already been properly explained in \cite[Remark 5.2]{baetal}. 

In fact, by \cite[Corollary 5.11]{toda}, there exists a non-trivial semi-orthogonal decomposition of $\DD^b(C^{(n)})$, where $C$ is a curve of genus $g \geq 2$, if $n \geq g$ (see also \cite[Theorem D]{bekr} and \cite[Corollary 3.10]{jile}). Note that this is what is expected by Conjecture \ref{conj1}, because $C^{(n)}$ is not minimal in the range $n \geq g \geq 2$.
\end{remark}

\section{Indecomposability in families: proof of Theorem  \ref{family}}

In this section we consider indecomposability of derived categories in certain families of varieties. Namely, taking  notations as in the Introduction, we assume that $f \colon \cX \rightarrow T$ is a smooth projective family of varieties,
 parametrized by a complex algebraic  variety $T$,  
such that the relative Albanese morphism for $f$, denoted by 
$$\varphi \colon \cX \rightarrow \Alb \cX / T,$$
 is a smooth morphism over $T$.

What we need to prove is the following 
\begin{proposition}\label{upper}
Under the above assumption,
the function
\[
u : T \rightarrow \mathbb{N} \cup \{- \infty \}, \quad t \mapsto u(t) := \dim \PBs| \omega_{\cX_t}| 
\]
is upper semi-continuous.
\end{proposition}
\begin{proof}
The Proposition can be proved quite similarly to \cite[Proposition 2.5]{baetal}. We give a sketch of it. Let us consider the exact sequence
\begin{equation}\label{exactPB}
\varphi^* \varphi_* \omega_f \otimes \omega_f^{\vee} \rightarrow \OO_{\cX} \rightarrow \OO_{\mathcal{PB}} \rightarrow 0,
\end{equation}
where $\omega_f := \det \Omega_{\cX / T}$ is the relative canonical line bundle of the smooth morphism $f \colon \cX \rightarrow T$, and $\mathcal{PB} \subseteq \cX$ is the closed subscheme  defined by the  ideal sheaf $\mathrm{Im}\bigl[\varphi^* \varphi_* \omega_f \otimes \omega_f^{\vee} \rightarrow \OO_{\cX}\bigr]$. Given any $t \in T$,  we can restrict \eqref{exactPB} to the fiber  $\cX_t$ of $f$ over $t$, obtaining the exact sequence
\begin{equation}\label{exactPB2}
a_{\cX_t}^* {a_{\cX_t}}_* \omega_{\cX_t} \otimes \omega_{\cX_t}^{\vee} \rightarrow \OO_{\cX_t} \rightarrow \OO_{\mathcal{PB}_t} \rightarrow 0.
\end{equation}
Indeed,  from the cartesian diagram
\[
\xymatrix{
\cX_t\, \ar@{^{(}->}[r] \ar[d]^-{\varphi_t} & \cX \ar[d]^-{\varphi} \\
(\Alb \cX / T)_t\, \ar@{^{(}->}[r]^-{\iota} &\Alb \cX / T
}
\]
we get that 
\begin{equation*}
(\varphi^* \varphi_* \omega_f )|_{\cX_t} = \varphi_t^* \iota^*\varphi_*\omega_f = \varphi_t^*   {\varphi_t}_*  \omega_{\cX_t} = a_{\cX_t}^* {a_{\cX_t}}_* \omega_{\cX_t},
\end{equation*}
where the first isomorphism follows by the commutativity of the diagram, and the last one by the definition of the relative Albanese morphism for $f$. The central isomorphism comes from 
the  base change of the smooth morphism $\varphi$ (see \cite[Lemma 1.3]{bondalorlov}, or \cite[p.\ 276]{poli}). More precisely, 
from the formula in \emph{op.cit.} we have
\[
\mathrm{L}\iota^* \mathrm{R}\varphi_*\omega_f = \mathrm{R}{\varphi_t}_* \omega_{\cX_t},
\]
where $\mathrm{L}$ and $\mathrm{R}$ denote the left and right derived functors, respectively.
By taking cohomology in degree $0$, we get $\mathrm{L}^0\iota^* (\mathrm{R}\varphi_*\omega_f) = {\varphi_t}_* \omega_{\cX_t}$. So we need to prove that
\begin{equation}\label{degss}
\mathrm{L}^0\iota^* (\mathrm{R}\varphi_*\omega_f) = \iota^* \varphi_*\omega_f.
\end{equation}
Note that, by definition and  the projection formula,
$\mathrm{R}\varphi_*\omega_f = \mathrm{R}\varphi_*(\omega_{\varphi} \otimes \varphi^* \omega_{\eta}) = \mathrm{R}\varphi_* \omega_{\varphi} \otimes  \omega_{\eta}$ 
(notations as in the diagram \eqref{notation0}).
Since $\varphi$ is a smooth morphism,  the higher direct images 
$\mathrm{R}^q\varphi_* \omega_{\varphi}$ 
are  locally free sheaves for all $q$  by Hodge theory (see, e.g., \cite[Theorem 10.21]{esvi}). Hence,  $\mathrm{R}^q\varphi_*\omega_f$ are locally free too, and $\mathrm{L}^p\iota^* (\mathrm{R}^q\varphi_*\omega_f) = 0$ for $p > 0$.
So the spectral sequence (see \cite[(3.10), p.\ 81]{huybrechts})
\[
E^{p, q}_2 = \mathrm{L}^p\iota^* (\mathrm{R}^q\varphi_*\omega_f) \Rightarrow \mathrm{L}^{p+q}\iota^* (\mathrm{R}\varphi_*\omega_f)
\]
degenerates, and we get \eqref{degss}.

Therefore, from \eqref{exactPB2} and Theorem \ref{mainthm}, we have that
 $\mathcal{PB}_t = \PBs|\omega_{\cX_t}|$ for any $t \in T$, and, since the fiber dimension of the proper morphism $\mathcal{PB} \rightarrow T$ is upper semi-continuous, this concludes the proof. 
\end{proof}

\begin{proof}[Proof of Theorem \ref{family}]
At this point,  the proof of Theorem \ref{family} is exactly the same of \cite[Theorem B]{baetal}, where one uses Proposition \ref{upper} instead of \cite[Proposition 2.5]{baetal} and Theorem \ref{Linthm} instead of \cite[Theorem 1.1]{baetal}.
\end{proof}

\providecommand{\bysame}{\leavevmode\hbox
to3em{\hrulefill}\thinspace}

\end{document}